\documentclass[10pt,a4paper,twoside,Times]{article}
\usepackage[all]{xy}
\usepackage{abstract}
\SelectTips{cm}{}
\usepackage{amsfonts,amssymb,amscd,amsmath,enumerate,verbatim,calc}

\usepackage{amscd,amssymb,amsopn,amsmath,amsthm,graphics,amsfonts,enumerate,verbatim,calc}
\usepackage[dvips]{graphicx}
\usepackage[colorlinks=true,linkcolor=blue,citecolor=blue]{hyperref}

\input xy
\xyoption{all}

\newcommand{\rt}{\rightarrow}
\newcommand{\lrt}{\longrightarrow}

\newcommand{\st}{\stackrel}

\newcommand{\CE}{\mathcal{E}}

\newcommand{\Mod}{{\rm{Mod}}}

\newcommand{\Prj}{{\rm{Proj}}}

\newcommand{\GPrj}{{\rm{GProj}}}
\newcommand{\GInj}{{\rm{GInj}}}

\newcommand{\KR}{\mathbb{K}(R)}

\newcommand{\KPR}{{\mathbb{K}({\Prj} \ R)}}
\newcommand{\KAPR}{{\mathbb{K}_{\rm{ac}}({\Prj} \ R)}}
\newcommand{\KTAPR}{{\mathbb{K}_{\rm{tac}}({\Prj} \ R)}}

\newcommand{\KPG}{{\mathbb{K}({\Prj} \ \Gamma)}}
\newcommand{\KAPG}{{\mathbb{K}_{\rm{ac}}({\Prj} \ \Gamma)}}
\newcommand{\KTAPG}{{\mathbb{K}_{\rm{tac}}({\Prj} \ \Gamma)}}

\newcommand{\KPS}{{\mathbb{K}({\Prj} \ S)}}

\newcommand{\KTAPS}{{\mathbb{K}_{\rm{tac}}({\Prj} \ S)}}

\newcommand{\op}{{\rm{op}}}

\newcommand{\Coker}{{\rm{Coker}}}
\newcommand{\Ker}{{\rm{Ker}}}

\newcommand{\Hom}{{\rm{Hom}}}
\newcommand{\Ext}{{\rm{Ext}}}

\newtheorem{theorem}{Theorem}[section]
\newtheorem{corollary}[theorem]{Corollary}
\newtheorem{lemma}[theorem]{Lemma}
\newtheorem{proposition}[theorem]{Proposition}

\theoremstyle{definition}

\newtheorem{example}[theorem]{Example}

\newtheorem{remark}[theorem]{Remark}

\theoremstyle{plain}

\theoremstyle{definition}

\numberwithin{equation}{section}

\begin{document}
\pagestyle{myheadings}
\markboth{H. Eshraghi, R. Hafezi, Sh. Salarian, Z.W. Li}{Gorenstein Projective Modules Over Triangular Matrix Rings}

\title{\bf  \large Gorenstein Projective Modules Over Triangular Matrix Rings\footnote{This research is supported in part by a grant from IPM (No.91130218)}}
\author{\bf H. Eshraghi \\
{\small eshraghi@sci.ui.ac.ir}\\
\\
{\bf R. Hafezi} \\
{\small r.hafezi@sci.ui.ac.ir} \\
\\
{\bf Sh. Salarian}\\
{\small salarian@ipm.ir}\\
{\small \it Department of Mathematics, University of Isfahan, P.O.Box: 81746-73441,}\\
 {\small \it  Isfahan, Iran and }\\
{\small \it School of Mathematics, Institute for Research in Fundamental Sciences (IPM), }\\
{\small \it P.O.Box: 19395-5746, Tehran, Iran.}\\
\\
\bf Z.W. Li\\
{\small zhiweili@sjtu.edu.cn}\\
{\small \it Department of Mathematics, Jiangsu Normal University, Xuzhou 221116,}\\
{\small \it  People's Republic of China}
}
\date{}

\setlength\absleftindent{0in }
\setlength\absrightindent{0in}
\maketitle

\begin{onecolabstract}
\noindent
{\bf Abstract.} Let $R$ and $S$ be Artin algebras and $\Gamma$ be their triangular matrix extension via a bimodule ${}_S M_R$. We study totally acyclic complexes of projective $\Gamma$-modules and obtain a complete description of Gorenstein projective $\Gamma$-modules. We then use this to construct some examples of Cohen-Macaulay finite and virtually Gorenstein triangular matrix algebras.
\end{onecolabstract}

\noindent
{\bf 2000 Mathematics Subject Classification:} 18G25, 16G10.

\noindent
{\bf\small Keywords:} Gorenstein projective, totally acyclic, Cohen-Macaulay finite.

\section{{\normalsize Introduction}}

Let $A$ be a not necessarily commutative ring with nonzero unity, $\Mod A$ (resp. mod$A$) be the category of all (resp. finitely generated) left $A$-modules, and $\Prj(A)$ (resp. {\rm proj}($A$)) be the full subcategory of all (resp. finitely generated) projective $A$-modules. Following \cite{EJ}, an $A$-module $X$ is said to be Gorenstein projective if it is a syzygy of a totally acyclic complex of projective $A$-modules, i.e. an exact complex of projective $A$-modules  which will be
left exact under applying $\Hom_A(-,\,P)$ for every projective $A$-module $P$. Throughout the paper $\GPrj(A)$ (resp. {\rm Gproj}($A$)) denotes the category of all (resp. finitely generated) Gorenstein projective $A$-modules; Gorenstein injective $A$-modules are defined dually.

\vspace{.05 cm}

Regardless of their role as a foundation of Gorenstein homological algebra, Gorenstein projective modules usually receive particular importance from several points of view. They are used in the theory of singularities (see e.g. \cite{Y} for a standard text) and in Tate cohomology of algebras (\cite{AM}, \cite{B}).  Also Cohen-Macaulay finite algebras, having recently been of particular attention (\cite{Bb}, \cite{C}, \cite{LZb}), are defined in terms of Gorenstein projective modules; recall that an Artin algebra $\Lambda$ is said to be of finite Cohen-Macaulay type (C.M. finite for short) provided it has only finitely many, up to isomorphism, indecomposable finitely generated Gorenstein projective modules, that is to say, the subcategory ${\rm Gproj}(\Lambda)$ is of finite type.

\vspace{.05 cm}

This article aims at studying Gorenstein projective modules over triangular matrix rings. Such modules have been studied in some special cases, see for example \cite{LZa}, \cite{XZ}, and \cite{Z}. The approach we will take here is to
study totally acyclic complexes of projective modules and regard Gorenstein projective modules as their syzygies. So throughout the rest of this paper, we fix a formal triangular matrix ring
\begin{equation*}
\Gamma = \left(
\begin{array}{ccc}
R & 0 \\
M & S
\end{array} \right)
\end{equation*}
where $R$ and $S$ are two associative not necessarily commutative rings with unity and  $_{S}M_R$ is an $S-R$ bimodule.
It is well-known that when $R$ and $S$ are $k$-Artin algebras, $k$ being a commutative artinian ring, and $M$ is finitely generated over $k$ which acts centrally on $M$, then $\Gamma$ is a $k$-Artin algebra. Therefore whenever $\Gamma$ is discussed from representation-theoretic aspects, we are assuming that it is an Artin algebra.

We then will use these information to construct some accessible examples of Cohen-Macaulay finite Artin algebras; namely, Cohen-Macaulay finite triangular matrix Artin algebras. In particular, we will deal with how $\Gamma$ inherits C.M. finiteness from $R$ and $S$ and vice versa. Finally, Gorenstein projective $\Gamma$-modules are used to study how virtually Gorensteinness may transfer from $R$ and $S$ to $\Gamma$; this is a concept which is closely related to the notion of Cohen-Macaulay finiteness.

\vspace{.2 cm}

\section{{\normalsize Gorenstein Projective Modules}}
Keep the notations of the previous section. Recall (e.g. from \cite{ARS}) that the category $\Mod\Gamma$ may be identified with the category $\mathfrak{C}$ consisting of all triples $(X,\,Y)_\varphi$ where $X\in\Mod R$, $Y\in\Mod S$ and $\varphi:M\otimes_R X\lrt Y$ is an $S$-linear map. Using this, we define two evaluation functors $e^1:\Mod\Gamma\lrt\Mod R$ and $e^2:\Mod\Gamma\lrt\Mod S$ in the following way: $e^1((X, Y)_\varphi)=X$ and $e^2((X, Y)_\varphi)=Y$, for every $\Gamma$-module $(X, Y)_\varphi$ and with the obvious rules over the morphisms. It is well-known that these functors both admit left
and right adjoints. The subscriptions $\lambda$ and $\rho$ are reserved respectively to denote the left and right adjoint functors.

Before defining these adjoint functors we need to recall that
any $\Gamma$-module could also be represented as $(X, Y)_\varphi$ where $X\in\Mod R$, $Y\in\Mod S$, and  $\varphi:X\lrt\Hom_S(M,\,Y)$ is a $R$-linear map. Hence for every $R$-module $X$, $e^1_\lambda(X)=(X,\,M\otimes_R X)_1$ and
$e^1_\rho(X)=(X,\,0)_0$. Similarly, for any $S$-module $Y$, $e^2_\lambda(Y)=(0,\,Y)_0$ and $e^2_\rho(Y)=(\Hom_S(M,\,Y),\, Y)_1$. More briefly, there are two adjoint pair of functors $(e^i,\,e^i_\rho)$ and $(e^i_\lambda,\,e^i)$, $i=1,2$.

\vspace{.1cm}

We need to extend these pairs to the corresponding homotopy categories. So we briefly recall that $\KR$, the homotopy category of $R$, has all complexes of $R$-modules as objects and its morphisms are the homotopy equivalence classes of the chain maps. $\KPR$ is the homotopy category formed by all complexes of projective $R$-modules and $\KAPR$ (resp. $\KTAPR$) is the full subcategory of $\KPR$ of all acyclic (resp. totally acyclic) complexes. Note that they are standard examples of triangulated categories \cite{N} and that extending the aforementioned adjoint pairs to the homotopy categories, a procedure that may be done in a natural way, leads in triangulated functors denoted by $k^1$, $k^2$ (the extended evaluation functors) and $k^1_\lambda$,
$k^1_\rho$, $k^2_\lambda$, and $k^2_\rho$ (their corresponding adjoints).

\vspace{.1cm}

The following result is well-known; see e.g. \cite[Proposition 2.1]{ET}.

\begin{lemma}\label{TPrj}
A $\Gamma$-module $(X, Y)_\varphi$ is projective if and only if the following statements are satisfied.
\begin{itemize}
\item[(i)] $\varphi:M\otimes_R X\lrt Y$ is a $S$-monomorphism.
\item[(ii)] $X\in\Prj(R)$ and $\Coker(\varphi)\in\Prj(S)$.
\end{itemize}
\end{lemma}

It should be pointed out that an analogous description for injective $\Gamma$-modules, which will be used in Theorem \ref{TGIG}, has been included in \cite{ET}.

\vspace{.2 cm}

Let $\CE\in\KPG$. According to the definitions, we represent $\CE$ as
$$ \CE:\quad\cdots\lrt(\CE_1^a,\,\CE_1^b)_{\varphi_1}\lrt (\CE_0^a,\,\CE_0^b)_{\varphi_0}\lrt (\CE_{-1}^a,\,\CE_{-1}^b)_{\varphi_{-1}}\lrt\cdots$$
where, for any integer $i$, $(\CE_i^a,\,\CE_i^b)_{\varphi_i}$ is a projective $\Gamma$-module. Note that, in view of the above lemma, there exists for any integer $i$ a split exact sequence $0\lrt M\otimes_R\CE_i^a\lrt\CE_i^b\lrt\Coker\varphi_i\lrt 0$. This induces a complex $$\cdots\lrt\Coker(\varphi_1)
\lrt\Coker(\varphi_0)\lrt\Coker(\varphi_{-1})\lrt\cdots$$ which we denote by $\Coker\CE$.

Before stating the following lemma, we recall that the symbol ${}^\perp$ is reserved to denote the (left and right) orthogonal classes with respect to the functor $\Ext^1(-, -)$. Also, ${\rm Add}({}_S M)$ contains all $S$-modules which are direct summands of arbitrary direct sums of copies of $M$.

\begin{lemma}\label{LTAPRSG}

Let $\CE\in\KPR$ and $\CE'\in\KPS$.
\begin{itemize}
  \item[(i)]  Assume that the functor $M\otimes_R-$ takes every acyclic complex of projective $R$-modules to an acyclic complex of $S$-modules.  Then  $\CE\in\KTAPR$ if and only if $k^1_\lambda(\CE)\in\KTAPG$.
  \item[(ii)] Assume ${\rm Add}({}_S M)\subseteq \GPrj(S)^\perp$. Then $\CE'\in\KTAPS$ if and only if $k^2_\lambda(\CE')\in\KTAPG$.
\end{itemize}
\end{lemma}

\begin{proof}
  Since the functors $e^1_\lambda$ and $e^2_\lambda$ preserve projective modules, the proof follows easily from Lemma \ref{TPrj} and the adjointness properties. Note that the assumption on $M$ in the first statement yields that  $k^1_\lambda(\CE)\in\KAPG$ whenever $\CE\in\KTAPR$. Likewise, in the second one, the assumption gives that when $\CE'\in\KTAPS$,then for any projective $R$-module $P$, the complex $\Hom_S(\CE', M\otimes_R P)$ is acyclic from which the desired result might be deduced in conjunction with Lemma \ref{TPrj}.
\end{proof}

As established in the above lemma, the following statements are crucial throughout the rest of this paper:

\vspace{.1 cm}

$(1)$ {\it The functor $M\otimes_R-$ takes every acyclic complex of projective $R$-modules to an acyclic complex of $S$-modules}.

\vspace{.1 cm}

$(2)$ {\it ${\rm Add}({}_S M)\subseteq \GPrj(S)^\perp=\{Y\in\Mod S: \Ext^1_S(G, Y)=0, \forall G\in\GPrj(S)$\}}.

\vspace{.1 cm}

\noindent We just refer to them by mentioning their assigned numbers.

\vspace{.1 cm}

The following proposition provides a complete description of totally acyclic complexes of projective $\Gamma$-modules.

\begin{proposition}\label{TTAP}
 Let the statements $(1)$ and $(2)$ hold and let $\CE\in\KPG$. Then $\CE\in\KTAPG$ if and only if $k^1(\CE)\in\KTAPR$ and $\Coker\CE\in\KTAPS$.
\end{proposition}
\begin{proof}
Assume first that $\CE\in\KTAPG$. Then the complexes $k^1(\CE)$ and $k^2(\CE)$ are acyclic. Since, by $(1)$, $M\otimes k^1(\CE)$ is acyclic, the acyclicity of $\Coker\CE$ follows. Then there exists a triangle
  \[\xymatrix{k_{\lambda}^1 k^1(\CE) \ar[r] & \CE \ar[r] & k^2_\lambda(\Coker\CE)
\ar@{~>}[r] & } \ \ \ (*) \] in $\KPG$. For any $Q\in\KPS$, it is clear that $k^2_\lambda(Q)\in\KPG$. Apply the homological functor $\Hom_{\KPG}(-,\,k^2_\lambda(Q))$
on $(*)$ to obtain the obvious long exact sequence. Since for any $R$-module $U$ and $S$-module $V$, $\Hom_{\KPG}(k_{\lambda}^1 k^1(U),\,k^2_\lambda(V))=0$, the assumption implies that $\Coker\CE\in\KTAPS$ so, by Lemma \ref{LTAPRSG}, $k^2_\lambda(\Coker\CE)\in\KTAPS$. The rest of the claim follows from the above triangle and another application
of Lemma \ref{LTAPRSG}; the converse may be proved using a similar argument.
\end{proof}

\begin{lemma}\label{LELGP}
  \begin{itemize}
   \item[(i)] If the statement $(1)$ holds, then for any Gorenstein projective $R$-module $U$, $e^1_\lambda(U)$ is a Gorenstein projective $\Gamma$-module.
    \item[(ii)]  If the statement $(2)$ holds, then for any Gorenstein projective $S$-module $V$, $e^2_\lambda(V)$ is a Gorenstein projective $\Gamma$-module.
  \end{itemize}
\end{lemma}
\begin{proof}
  Is a direct consequence of Lemma \ref{LTAPRSG}.
\end{proof}

We are now in the position to prove the main result, which is a complete description of Gorenstein projective $\Gamma$-modules.

\vspace{.1 cm}

\begin{theorem}\label{TGPG}
 Suppose that both of the statements $(1)$ and $(2)$ are satisfied. Then a $\Gamma$-module $(X,\,Y)_\varphi$ is Gorenstein projective if and only if
  \begin{itemize}
    \item[(a)] $X$ and $\Coker\varphi$ are Gorenstein projective respectively as $R$ and $S$-modules.
    \item[(b)] $\varphi$ is a monomorphism.
\end{itemize}
\end{theorem}
\begin{proof} The sufficiency follows from the exact sequence $$0\lrt (X,\,M\otimes_R X)_1\lrt (X,\,Y)_\varphi\lrt (0,\,\Coker\varphi)_0\lrt 0\quad (**)$$ in conjunction
  with Lemma \ref{LELGP}. If, conversely, $(X,\,Y)_\varphi$ is a Gorenstein projective $\Gamma$-module, then it is, say, the 0-th syzygy of some $\CE\in\KTAPG$. Note
  that by Proposition \ref{TTAP}, $\Coker\CE\in\KTAPS$ and $k^1(\CE)\in\KTAPR$. Moreover, $X=\ker(k^1(\CE)_0\rt k^1(\CE)_{-1})$ and $\Coker\varphi=\ker((\Coker\CE)_0\rt(\Coker\CE)_{-1})$.   Hence the statement $(a)$ follows. That $\varphi$ is a monomorphism follows because $M\otimes k^1(\CE)$ is an exact complex.
\end{proof}

In the sequel, we will proceed in the reverse direction. Namely, we will examine how close we are to the statements $(1)$ and $(2)$ by knowing that the above description for all Gorenstein projective $\Gamma$-modules is valid.

\begin{remark}\label{RCGP}
Suppose that all Gorenstein projective $\Gamma$-modules can be classified  as mentioned in the above theorem.
We first show that $(2)$ holds.  To this end, it is enough to prove that for any index set $I$,
$\oplus_I M \in {\rm (GProj(S))^\perp}$ or, equivalently, any short exact sequence $$0 \lrt\oplus_I M \st{\alpha}\lrt D \st{\beta}\lrt G \lrt 0$$ with $G$ Gorenstein projective $S$-module splits. From the hypothesis, it follows that $(\oplus_I R, D)_\alpha$ is a Gorenstein projective $\Gamma$-module. So it is the 0-th syzygy of some $\CE\in\KTAPG$. Hence we get a short exact sequence $$0 \lrt (\oplus_I R, D)_\alpha \st{f}\lrt \CE_0 \st{g}\lrt (K_1,K_2)_\gamma \lrt 0$$ in which $(K_1, K_2)_\gamma$ is also a Gorenstein projective $\Gamma$-module. In view of the fact that $K_1$ should be Gorenstein projective as $R$-module, one infers that the sequence $0 \lrt \oplus_I R\st{f^a} \lrt \CE^a_0 \st{g^a}\lrt K_1 \lrt 0$ splits. Now apply Lemma \ref{TPrj} and the commutative exact diagram

\[ \xymatrix{
0 \ar[r] & \oplus_I M \ar[r]^{1_M\otimes f^a} \ar[d]^{\alpha} &  M\otimes \CE^a_0 \ar[r]^{1_M\otimes g^a} \ar[d]^{\varphi} &
M\otimes K_1 \ar[r] \ar[d]^{\gamma} & 0  \\
0\ar[r] & D\ar[r]^{f^b} & \CE^b_0 \ar[r]^{g^b} & K_2 \ar[r] & 0 }\]

\noindent to show that the sequence at the beginning of the argument is split.

Next, we prove that ${\rm Tor}_1^R(M,\,P)=0$ for any Gorenstein projective $R$-module $P$. So let $P$ be such a module. Then, by the hypothesis, $(P,\,M\otimes_R P)_1$  is Gorenstein projective and so is the 0-th syzygy of some $\CE\in\KTAPG$. Hence we get a short exact sequence
$$0\lrt(K_1,\, K_2)_\gamma\lrt \CE_{-1}\lrt (P,\,M\otimes_R P)_1\lrt 0$$ with $(K_1,\,K_2)_\gamma$ Gorenstein projective. This gives a commutative exact diagram

\[\xymatrix{
 & M\otimes_R K_1\ar[r]^{\omega}\ar[d]^\gamma & M\otimes_R \CE_{-1}^a\ar[r]\ar[d] & M\otimes_R P\ar[r]\ar[d]^{=} & 0\\
0\ar[r] & K_2\ar[r] & \CE_{-1}^b\ar[r] & M\otimes_R P\ar[r] & 0.
    } \]
Since, by the assumption, $\psi$ is a monomorphism it follows that $\omega$ is a monomorphism, implying that ${\rm Tor}_1^R(M,\,P)=0$. It is straight forward to check that this is equivalent to saying that $M\otimes_R-$ preserves every totally acyclic complex of projective $R$-modules.
\end{remark}

\vspace{.1 cm}

Recall that the ${\rm T}_2$-extension of a ring $R$ is given by  \begin{equation*}
{\rm T}_2(R) = \left(
\begin{array}{ccc}
R & 0 \\
R & R
\end{array} \right)
\end{equation*}
and that every module over ${\rm T}_2(R)$ is a homomorphism $X\st{\varphi}\lrt Y$ of $R$-modules. We include the following corollary of Theorem \ref{TGPG}.

\begin{corollary} A ${\rm T}_2(R)$-module $X\st{\varphi}\lrt Y$ is Gorenstein projective if and only if $X$ and $\Coker \varphi$ are both Gorenstein projective $R$-modules and $\varphi$ is a monomorphism.
\end{corollary}

It is worth noting that the dual version of the above arguments may be applied to deduce a corresponding classification for Gorenstein injective modules over $\Gamma$. So we
skip the proofs and record the dual version of Theorem \ref{TGPG}. Just note that the following conditions which are, in some sense, dual to $(1)$ and $(2)$ should be fulfilled.

\vspace{.1 cm}

$(3)$ {\it The functor $\Hom_S(M, -)$ takes every acyclic complex of injective $S$-modules to an acyclic complex of $R$-modules}.

\vspace{.1 cm}

$(4)$ {\it $\Hom_S(M, I)\in{}^\perp{\rm GInj}(R)$, for all injective $S$-module $I$}.

\begin{theorem}\label{TGIG}
Suppose that the statements $(3)$ and $(4)$ are satisfied. Then a $\Gamma$-module $(X,\, Y)_\varphi$ is Gorenstein injective if and only if
\begin{itemize}
\item[(a)] $Y$ and $\ker\varphi$ are Gorenstein injective respectively as $S$ and $R$-modules.
\item[(b)] $\varphi:X\lrt\Hom_S(M,\,Y)$ is an epimorphism.
\end{itemize}
\end{theorem}

\begin{corollary} A ${\rm T}_2(R)$-module $X\st{\varphi}\lrt Y$ is Gorenstein injective if and only if $Y$ and $\Ker \varphi$ are both Gorenstein injective $R$-modules and $\varphi$ is an epimorphism.
\end{corollary}

\section{{\normalsize C.M. Finiteness And Virtually Gorensteinness}}

 {\it 3.1. C.M. finiteness}. We will now deal with how C.M. finiteness may transfer from the algebras $R$ and $S$ to the triangular matrix algebra $\Gamma$. We start with the following easy observation. Assume that $(1)$ and $(2)$ are satisfied and $\Gamma$ is of finite Cohen-Macaulay type. For an indecomposable finitely generated Gorenstein projective $R$-module $X$, it is routine to verify that $e^1_\lambda(X)=(X,\,M\otimes_R X)_{1}$ is an indecomposable finitely generated $\Gamma$-module which is Gorenstein projective by Lemma \ref{LELGP}. Since $\Gamma$ is C.M. finite one has only finitely many, up to isomorphism, choices for $X$, that is, $R$ should be of finite Cohen-Macaulay type. A similar argument applies to deduce that $S$ is also of finite Cohen-Macaulay type.

Therefore $\Gamma$ C.M. finite implies $R$ and $S$ C.M. finite while there are examples for which the converse fails. For, let $k$ be an algebraically closed field and consider the self injective $k$-algebra $\Lambda_t=k[x]/(x^t)$ with $t>5$. By \cite[Example 4.17]{Bb}, the ${\rm T}_2$-extension of $\Lambda_t$ \begin{equation*} {\rm T}_2(\Lambda_t)=\left(
\begin{array}{ccc}
\Lambda_t& 0 \\
\Lambda_t & \Lambda_t
\end{array} \right)
\end{equation*}
is of infinite Cohen-Macaulay type while $\Lambda_t$ itself is an algebra even of finite representation type.

\vspace{.1cm}

However the material prepared in the previous section makes it possible to overcome this defect in some cases. Recall that an Artin algebra $\Lambda$ is said to be {\it Cohen-Macaulay free} if ${\rm Gproj}(\Lambda)={\rm proj}(\Lambda)$.

\begin{proposition}\label{TMR1} Assume $(1)$ and $(2)$ hold.
\begin{itemize}
\item[(i)]  Let $R$ be C.M. free. Then $\Gamma$ is of finite Cohen-Macaulay type if and only if so is $S$.
\item[(ii)] Let $S$ be C.M. free. Then  $\Gamma$ is of finite Cohen-Macaulay type if and only if so is $R$.
\item[(iii)] $\Gamma$ is C.M. free if and only if so are $R$ and $S$.
\end{itemize}
\end{proposition}

\begin{proof}
$(i)$ Let $(X,\,Y)_\varphi$ be an indecomposable finitely generated Gorenstein projective $\Gamma$-module. Consider the short exact sequence $(**)$ (in the proof of Theorem \ref{TGPG}) in $\Mod\Gamma$.
One may apply Theorem \ref{TGPG} to deduce that $X\in{\rm Gproj}(R)={\rm proj}(R)$, that is, $(X,\, M\otimes_R X)_1$ is a finitely generated projective $\Gamma$-module by \ref{TPrj}. On the other hand, it follows from Theorem \ref{TGPG} that $(0,\,\Coker\varphi)_0\in{\rm Gproj}(\Gamma)$, i.e., the sequence $(**)$ splits. Since  $(X,\,Y)_\varphi$ is indecomposable, either  $(X,\,Y)_\varphi=(0,\,\Coker\varphi)_0$ or $(X,\,Y)_\varphi=(X,\,M\otimes_R X)_1$. It follows that in
the first case $\Coker\varphi$ is an indecomposable finitely generated Gorenstein projective $S$-module while in the second one $X$ is an indecomposable finitely generated projective $R$-module. This gives the sufficiency; the
necessity follows from the above remarks and the proof of $(ii)$ is similar. Also $(iii)$ can be deduced directly from Lemma \ref{TPrj} and Theorem \ref{TGPG}.
\end{proof}

\vspace{.1cm}

Note that important examples arise when $R$ or $S$ are algebras of finite global dimension since for such algebras projectives and Gorenstein projectives coincide.

The following example illustrates an application of Theorem \ref{TGPG}.

\example\label{EALL}
Let $A_n$ be the quiver $v_1\rt\cdots\rt v_n$. Then $RA_n$, the path algebra of $A_n$ over the Artin algebra $R$, is the lower triangular $n\times n$ matrix algebra defined inductively as
\begin{equation*}
T_n(R)=\left(
\begin{array}{ccc}
R &  0 \\
N & {\rm T}_{n-1}(R)
\end{array} \right)
\end{equation*}
 where $N=R\oplus R\cdots\oplus R$ ($n-1$ copies), as a right $R$-module, whose left $RA_{n-1}={\rm T}_{n-1}(R)$-structure corresponds to the representation
 $R\rt\cdots\rt R$ of $A_{n-1}$ and ${\rm T}_1(R)=R$. Since $N$ is projective both as $R$- and ${\rm T}_{n-1}(R)$-module, from Theorem \ref{TGPG} we deduce that any Gorenstein projective ${\rm T}_n(R)$-module is of the form $(X,\,Y)_\varphi$ where $X\in\GPrj(R)$, $\varphi:N\otimes_R X\lrt Y$ is a monomorphism with $\Coker\varphi\in\GPrj({\rm T}_{n-1}(R))$. Assume now that ${\rm T}_{n-1}(R)$ is of finite Cohen-Macaulay type for some $n\geq 1$ and pick a module $Z$ with ${\rm add }(Z)={\rm Gproj}({\rm T}_{n-1}(R))$. Set $\Lambda={\rm End}(Z)^{\op}$. Following \cite{LZa}, we let $\mathcal{C}$ be the subcategory of the morphism category of ${\rm proj}(\Lambda)$ consisting of all monomorphisms $f:P_2\lrt P_1$ with $\Coker f\in {\rm proj}^{\leq 1} \Omega(\Lambda)$ where the latter is the subcategory of all torsionless modules of projective dimension at most 1; see pages 1809-1811 of \cite{LZa} for the details.
 Moreover, put $\mathcal{C'}$ be the subcategory of $\mathcal{C}$ consisting of all $f:P_2\lrt P_1$ satisfying the extra assumption $P_2=\Hom_{T_{n-1}(R)}(Z,\, N\otimes_R X)$ where $X$ is a Gorenstein projective $R$-module; it is closed under direct summands and the functor $\Hom_{{\rm T}_{n-1}(R)}(Z,\,-)$ provides an embedding ${\rm Gproj}({\rm T}_{n}(R))\hookrightarrow \mathcal{C'}$. Therefore by the proof of \cite[Theorem 1.2]{LZa} if ${\rm proj}^{\leq 1} \Omega(\Lambda)$ is of finite type, the subcategory $\mathcal{C'}$, and so ${\rm Gproj}({\rm T}_n(R))$  is of finite type, that is, ${\rm T}_n(R)$ is of finite Cohen-Macaulay type. This provides an inductive procedure to relate C.M. finiteness of ${\rm T}_n(R)$ to that of ${\rm T}_m(R)$ for $m<n$.

\vspace {.2 cm}

{\it 3.2. Virtually Gorensteinness}. Recall that the concept of virtually Gorenstein algebras was defined in \cite{BR} as a natural generalization of Gorenstein algebras. By definition, an Artin algebra $\Lambda$ is said to be virtually Gorenstein provided $\GPrj(\Lambda)^\perp={}^\perp\GInj(\Lambda)$ where $\GInj(\Lambda)$ is the subcategory of all Gorenstein injective $\Lambda$-modules and the orthogonal classes are defined with respect to $\Ext^1$. Such algebras have been illuminated to be in close relationship with C.M. finite algebras; see \cite{Bb} for more information.
 Theorem \ref{TGPG} then prepares a good tool to prove the following theorem.

\begin{proposition}\label{TMR2}
Assume that the statements $(1)$, $(2)$, $(3)$, and $(4)$ are all true. Then $\Gamma$ is virtually Gorenstein if and only if $R$ and $S$ are.
\end{proposition}
\begin{proof}
 For a $\Gamma$-module $(X,\,Y)_\varphi$, we claim that $(X,\,Y)_\varphi\in\GPrj(\Gamma)^\perp$ if and only if $X\in\GPrj(R)^\perp$ and $Y\in\GPrj(S)^\perp$. To see this, assume first
 that $(X,\,Y)_\varphi\in\GPrj(\Gamma)^\perp$ and $U$ is a Gorenstein projective $R$-module. By \ref{LELGP}, $e^1_\lambda(U)=(U,\,M\otimes_R U)_1\in\GPrj(\Gamma)$. Therefore $0=\Ext^1_\Gamma(e^1_\lambda(U),\,(X,\,Y)_\varphi)\simeq\Ext^1_R(U,\,X)$ implying that  $X\in\GPrj(R)^\perp$. (Note that the latter isomorphism follows by considering a projective resolution of $U$ and applying the functor $e^1_\lambda$, which is exact in view of statement ($1$), to obtain a projective resolution of $e^1_\lambda(U)$.) The other statement follows similarly. The converse is an immediate consequence of the adjoint isomorphisms and the
 short exact sequence $(**)$ (in the proof of Theorem \ref{TGPG}) which is valid for every Gorenstein projective $\Gamma$-module. One may apply a dual argument and Theorem \ref{TGIG} to elicit a similar statement concerning ${}^\perp\GInj(\Gamma)$. These, in particular, yield that $\Gamma$ is virtually Gorenstein if and only if $R$ and $S$ are.
\end{proof}

\example
Let $R$ be an algebra. Clearly the ${\rm T}_2$-extension of $R$, ${\rm T}_2(R)$, satisfies the hypothesis of the above proposition. Hence, ${\rm T}_2(R)$ is virtually Gorenstein if and only if so is $R$. In particular, ${\rm T}_2(R)$ is virtually Gorenstein if $R$ is of finite representation type.

\section*{Acknowledgments} We would like to thank the referee for her/his useful comments. Sh. Salarian thanks the Center of Excellence for Mathematics (University of Isfahan).
H. Eshraghi and R. Hafezi thank the Institute for Research in Fundamental Sciences (IPM).

\end{document}